\documentclass{article}

\usepackage[utf8]{inputenc}
\usepackage{array}
\usepackage{multirow}
\usepackage{amsfonts,amsmath, amsthm}
\usepackage{comment}
\usepackage{amssymb}
\usepackage{graphicx}

\usepackage[symbol*]{footmisc}

\usepackage[usenames,dvipsnames,table]{xcolor}

\usepackage{tabularx}
\newcolumntype{L}[1]{>{\raggedright\arraybackslash}p{#1}} 	
\newcolumntype{C}[1]{>{\centering\arraybackslash}p{#1}} 	
\newcolumntype{R}[1]{>{\raggedleft\arraybackslash}p{#1}} 	

\title{Computing Maximal and Minimal Trap Spaces of Boolean Networks}

\author{Hannes Klarner\footnotemark, Alexander Bockmayr and Heike Siebert}
\date{September 2015}

\newcommand{\StateSpace}{S}
\newcommand{\Transition}{\rightarrow}
\newcommand{\SyncTransition}{\twoheadrightarrow}
\newcommand{\AsyncTransition}{\hookrightarrow}

\newcommand{\SteadyStates}[1]{S_{#1}}
\newcommand{\SubSpaces}{S^{\Free}}
\newcommand{\TrapSpaces}[1]{S^{\Free}_{#1}}
\newcommand{\MinTrapSpaces}[1]{\min(\TrapSpaces{#1})}
\newcommand{\MaxTrapSpaces}[1]{\max(\TrapSpaces{#1})}
\newcommand{\PrimeImplicants}{\mathcal P}

\newcommand{\Free}{\ensuremath{\star}}
\newcommand{\SmallestSubspace}[1]{\mathit{Sub(#1)}}

\newcommand{\SW}{5pt}
\newcommand{\Pattern}[4]{\makebox[\SW][c]{#1}\makebox[\SW][c]{#2}\makebox[\SW][c]{#3}\makebox[\SW][c]{#4}}

\let\emptyset\varnothing
\newtheorem{theorem}{Theorem}
\newtheorem{corollary}{Corollary}

\begin{document}

\maketitle

\begin{abstract}
Asymptotic behaviors are often of particular interest when analyzing Boolean networks that represent biological systems such as
signal transduction or gene regulatory networks.
Methods based on a generalization of the steady state notion, the so-called \emph{trap spaces},
can be exploited to investigate attractor properties as well as for model reduction techniques.
In this paper, we propose  a novel optimization-based method for computing all minimal and maximal trap spaces and motivate their use.
In particular, we add a new result yielding a lower bound for the number of cyclic attractors and illustrate the methods with a study of a MAPK pathway model.
To test the efficiency and scalability of the method, we compare the performance of the ILP solver \textsc{gurobi} with the ASP solver \textsc{potassco} in a benchmark of random networks.
\end{abstract}

\footnotetext{$^*$Corresponding author: hannes.klarner@fu-berlin.de, Freie Universität Berlin, FB Mathematik und Informatik, Arnimallee 6, 14195 Berlin, Germany.}

\section{Introduction}

Boolean network models have long since proved their worth in the context of modeling complex biological systems \cite{Albert2012BooleanReview}.
Of particular interest are number and size of attractors as well as their locations in state space since these properties often relate well to important biological behaviors.

In this paper, we explore the notion of \emph{trap space}, which is a subspace of state space that no path can leave, for model reduction and attractor analysis.
After providing the relevant terminology, we demonstrate that trap spaces can be used for model reduction as well as to predict the number and locations of the system's attractors.
We then introduce the \emph{prime implicant graph} as an object which captures the essential dynamical information required to compute trap spaces of a Boolean network.

In practice, methods that exploit trap spaces can only be useful if their identification scales efficiently with the size of the network.
We provide an \emph{Integer Linear Programming} (ILP) and an \emph{Answer Set Programming} (ASP) formulation and present the results of a benchmark that compares the performances of
the ILP solver \textsc{gurobi} and the ASP solving collection \textsc{potassco}.
Finally, we apply our methodology to a MAPK pathway model.

This paper is an extended version of \cite{Klarner2014Computing}.
To allow for a more intuitive understanding of technical terms,
we decided to reformulate some of the previously presented notions in terms of the state space of Boolean networks.
The extension consists of discussing both maximal and minimal trap spaces, an ASP formulation of the optimization problem, the benchmark and an extension of the MAPK case study.

\subsection{Background}\label{Background}

We consider variables from the Boolean domain $\mathbb B=\{0,1\}$ where $1$ and $0$ represent the truth values \emph{true} and \emph{false}.
A Boolean \emph{expression} $f$ over the variables $V=\{v_1,\dots,v_n\}$ is defined by a formula over the grammar
\[f::= 0 \mid 1 \mid v \mid \overline f \mid  f_1\cdot f_2 \mid f_1+f_2\]
where $v\in V$ signifies a variable, $\overline f$ the negation, $f_1\cdot f_2$ the conjunction and $f_1+f_2$ the (inclusive) disjunction of the expressions $f,f_1$ and $f_2$.
Given an \emph {assignment} $x:V\rightarrow\mathbb B$, an expression $f$ can be evaluated to a value $f(x)\in\mathbb B$ by substituting the values $x(v)$ for the variables $v \in V$.
If $f(x)=f(y)$ for all assignments $x,y: V \rightarrow\mathbb B$, we say $f$ is \emph{constant} and write $f=c$, with $c\in\mathbb B$ being the constant value.
A \emph{Boolean network} $(V,F)$ consists of $n$ variables $V=\{v_1,\dots,v_n\}$ and $n$ corresponding Boolean expressions $F=\{f_1,\dots,f_n\}$ over $V$.
In this context, an assignment $x:V\rightarrow\mathbb B$ is also called a \emph{state} of the network and the \emph{state space} $\StateSpace$ consists of all possible $2^n$ states.
We specify states by a sequence of $n$ values that correspond to the variables in the order given in $V$, i.e.,
$x=110$ should be read as $x(v_1)=1, x(v_2)=1$ and $x(v_3)=0$.
The expressions $F$ can be thought of as a function $F:\StateSpace\rightarrow \StateSpace$ governing the network behavior.
The \emph{image} $F(x)$ of a state $x$ under $F$ is defined to be the state $y$ that satisfies $y(v_i)=f_i(x)$.
To illustrate these concepts we introduce a running example in Fig.~\ref{fig:01}.

\begin{figure}
 \begin{align*}
  f_1 &= v_1+v_2 			& f_1(1101) &= 1+1=1 			& F(1101) &=1101 \\
  f_2 &= v_1\cdot v_4 			& f_2(1101) &= 1\cdot 1=1 		& F(0000) &=0001 \\
  f_3 &= \overline{v_1}\cdot v_4 	& f_3(1101) &= \overline{1}\cdot 1=0 	& F(0110) &=1000 \\
  f_4 &= \overline{v_3} 		& f_4(1101) &= \overline{0}=1 		& F(1111) &=1100
 \end{align*}
 \caption{
 (left)  An example Boolean network with four variables.
 (middle) An example for how the Boolean expressions are evaluated for a given state $x=1101$.
 (right) Some examples for the image $F(x)$ of a state $x\in\StateSpace$.
 }\label{fig:01} 
\end{figure}

The \emph{state transition graph} of a Boolean network $(V,F)$ is the directed graph $(\StateSpace,\rightarrow)$ where
the transitions $\rightarrow\;\subseteq \StateSpace\times \StateSpace$ are obtained from $F$ via a given update rule.
We mention two update rules here, the \emph{synchronous rule} and its transition relation $\SyncTransition \;\subseteq \StateSpace\times \StateSpace$,
and the \emph{asynchronous rule} and its transition relation $\AsyncTransition \;\subseteq \StateSpace\times \StateSpace$.
The former is defined by $x\SyncTransition y$ iff $F(x)=y$.
To define $\AsyncTransition$ we need the Hamming distance $\Delta: \StateSpace\times \StateSpace\rightarrow\{1,\dots,n\}$ between states which is given by $\Delta(x,y):=|\{v\in V \mid x(v)\neq y(v)\}|$.
We define $x\AsyncTransition y$ iff either $x=y$ and $F(x)=x$ or $\Delta(x,y)=1$ and $\Delta(y,F(x))<\Delta(x,F(x))$.

A path in $(\StateSpace,\rightarrow)$ is a sequence of states $(x_1,\dots,x_{l+1})$ with $x_i\rightarrow x_{i+1}$, for $1\leq i\leq l$.
An  non-empty set $T\subseteq \StateSpace$ is a \emph{trap set} w.r.t. $\rightarrow$ if for every $x\in T$ and $y\in \StateSpace$ with $x\rightarrow y$ it holds that $y\in T$.
An inclusion-wise minimal trap set is also called an \emph{attractor} of $(\StateSpace,\rightarrow)$.
Note that every trap set contains at least one minimal trap set and therefore at least one attractor.
A variable $v\in V$ is \emph{steady} in an attractor $A\subseteq\StateSpace$ iff $x(v)=y(v)$ for all $x,y\in A$ and \emph{oscillating} otherwise.
We distinguish two types of attractors depending on their size.
If $A\subseteq \StateSpace$ is an attractor and $|A|=1$ then $A$ is called a \emph{steady state} and if $|A|>1$ we call it a \emph{cyclic attractor}.
The cyclic attractors of $(\StateSpace, \SyncTransition)$ are, in general, different from the cyclic attractors of $(\StateSpace, \AsyncTransition)$.
The steady states, however, are identical in both transition graphs because $x\Transition x$ iff $F(x)=x$ in both cases.
Hence, we may omit the update rule and denote the set of steady states by $\SteadyStates F$.
The state transition graphs of the running example is given in Fig.~\ref{fig:02}.

\section{Methods}
\subsection{Trap Spaces}
A \emph{subspace of $\StateSpace$} is characterized by its fixed and free variables.
It may be specified by a mapping $p:U\rightarrow\mathbb B$ where $U\subseteq V$ is the subset of fixed variables,
$p(u)$ the value of $u\in U$ and the remaining variables, $V\setminus U$, are said to be free.
Subspaces are sometimes referred to as \emph{symbolic states} or \emph{partial states}.
We specify subspaces like states but allow in addition the symbol "$\Free$" to indicate that a variable is free, i.e.,
$p=\Pattern{\Free}{\Free}10$ means $U=\{v_3,v_4\}$ and $p(v_3)=1, p(v_4)=0$.
The set $\SubSpaces=\SubSpaces(V)$ denotes all possible $3^n$ subspaces.
Note that states are a special case of subspace and that $\StateSpace\subset\SubSpaces$ holds.
We denote the \emph{fixed variables} $U$ of $p\in\SubSpaces$ by $U_p\subseteq V$.
A subspace $p$ references the states $\StateSpace[p]:=\{x\in \StateSpace \mid \forall v\in U_p: x(v)=p(v)\}$.

A \emph{trap space} is a subspace that is also a trap set.
Trap spaces are related to the \emph{seeds} in \cite{Siebert2011SymbolicSteadyStates} and \emph{stable motifs} in \cite{Zanudo2013StableMotifs}.
Trap spaces are therefore trap sets with a particularly simple structure.
They generalize the notion of steadiness from states to subspaces.
We denote the trap spaces of $(\StateSpace,\Transition)$ by $\TrapSpaces{\Transition}$.
Since every trap set contains at least one attractor,
inclusion-wise minimal trap spaces can be used to predict the location of a particular attractor while maximal trap spaces may be useful in understanding the
commitment of the system to a set of attractors.
Hence, we define a partial order on $\SubSpaces$ based on whether the referenced subspaces are nested: $p\leq q$ iff $\StateSpace[p]\subseteq \StateSpace[q]$.
The \emph{minimal} and \emph{maximal} trap spaces are defined by
\begin{alignat*}{2}
 \MinTrapSpaces{\Transition}:=&\{p\in\TrapSpaces{\Transition}\mid \nexists\; q\in\TrapSpaces{\Transition}:\; q<p\}\\
 \MaxTrapSpaces{\Transition}:=&\{p\in\TrapSpaces{\Transition}\mid \nexists\; q\in\TrapSpaces{\Transition}:\; (q>p) \wedge (\StateSpace[q]\neq \StateSpace)\}
\end{alignat*}
where we use the condition $\StateSpace[q]\neq \StateSpace$ because otherwise $S$ is a priori the (unique) maximal trap space of any network which is only informative in
the special case that the network has a single attractor.
The minimal and maximal trap spaces of the running example are illustrated in Fig.~\ref{fig:02}.
Note that in this case, for each attractor $A\subseteq\StateSpace$ of $(\StateSpace,\AsyncTransition)$ there is $p\in\MinTrapSpaces{\AsyncTransition}$ s.t. $\StateSpace[p]=A$.

\begin{figure}
 \centering
 \makebox[\textwidth][c]{\hspace{2.5cm}\includegraphics[width=13cm]{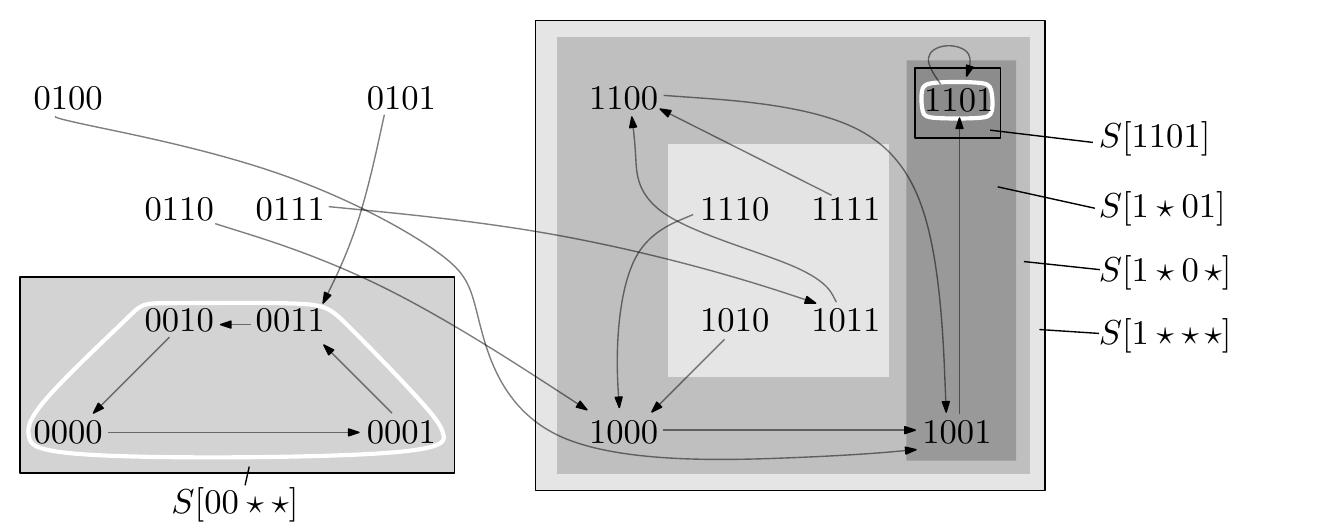}}
 
 (a) The asynchronous transition graph $(S,\AsyncTransition)$
 
 \vspace{3mm}
 
 \makebox[\textwidth][c]{\hspace{2.5cm}\includegraphics[width=13cm]{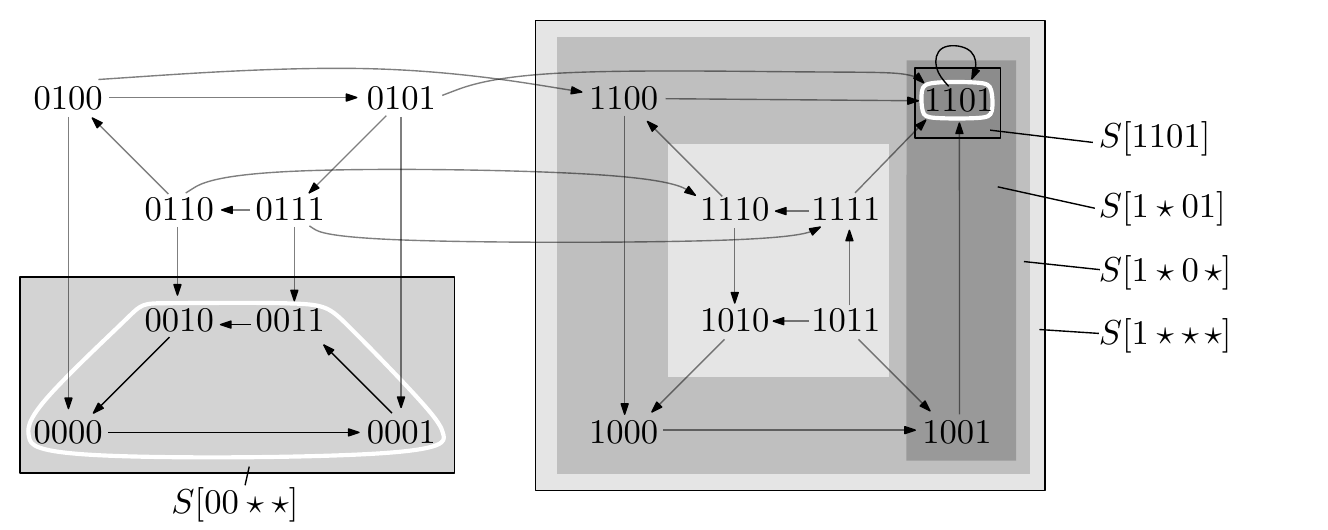}}
 
 (b) The synchronous transition graph $(S,\SyncTransition)$
 
 \caption{
 The synchronous and asynchronous transition graphs of the running example.
 The trap spaces (and the way they are nested) are given as a heat map.
 Minimal and maximal trap spaces are outlined by a black line and the attractors by a white line.
 Using the "$\Free$" notation we get the following sets:
 $\TrapSpaces{\Transition}=\{ \protect\Pattern{\Free}{\Free}{\Free}{\Free},%
				    \protect\Pattern{0}{0}{\Free}{\Free},%
				    \protect\Pattern{1}{\Free}{\Free}{\Free},%
				    \protect\Pattern{1}{\Free}{0}{\Free},
				    \protect\Pattern{1}{\Free}{0}{1},
				    \protect\Pattern{1}{1}{0}{1}\}$,
 $\MinTrapSpaces{\Transition}=\{\protect\Pattern{0}{0}{\Free}{\Free},%
				    \protect\Pattern{1}{1}{0}{1}\}$
 and				    
 $\MaxTrapSpaces{\Transition}=\{\protect\Pattern{0}{0}{\Free}{\Free},%
				    \protect\Pattern{1}{\Free}{\Free}{\Free}\}$
 and $\SteadyStates F=\{1101\}$.
 }\label{fig:02}
\end{figure}

\subsection{Characterization of Trap Spaces}
Analogous to the characterization of steady states by the equation $F(x)=x$ we show that trap spaces can be characterized by the inequality $p\geq F[p]$.

To do so, we first define $f[p]$ to be the expression obtained by substituting the values $p(v)$ for the variables $v\in U_p$ in $f$.
The \emph{image} $F[p]$ of a subspace $p$ under $F=\{f_1,\dots,f_n\}$ is then the subspace $q$ with
$U_q:=\{v_i\in V \mid f_i[p] \text{ is constant} \}$ and $q(v_i):=f_i[p]$, for all $v_i \in U_q$.
The following theorem characterizes trap spaces.
\begin{theorem}\label{thm:1}
 A subspace $p$ is a trap set of $(S,\Transition)$ if and only if $p\geq F[p]$.
\end{theorem}
\begin{proof}
 $p\in\SubSpaces$ is a trap set of $(S,\Transition)$ iff there are no $x\in \StateSpace[p]$ and $y\in S\setminus\StateSpace[p]$ s.t. $x\Transition y$.
 This is equivalent to $f_i(x)=p(v_i)$ for all $x\in\StateSpace[p]$ and $v_i\in U_p$, which is equivalent to $p\geq F[p]$.
\end{proof}

Note that the proof of Thm.~\ref{thm:1} is identical for the synchronous and asynchronous transition graphs.
As a corollary we get that $\TrapSpaces{\AsyncTransition}=\TrapSpaces{\SyncTransition}$, i.e,
that trap spaces are (like steady states) independent of the update rule.
We therefore write $\TrapSpaces{F}$ instead of $\TrapSpaces{\Transition}$.
For our running example, note that the trap spaces of $(S,\AsyncTransition)$ and $(S, \SyncTransition)$ are indeed identical, see Fig.~\ref{fig:02}.
Note also that $\SteadyStates F\subseteq\MinTrapSpaces F$ and so calculating all minimal trap spaces yields, in addition to the information on cyclic attractors, also all steady states.
Example calculations are given in Fig.~\ref{fig:03}.

\begin{figure}
 \begin{align*}
   p :&= \Pattern{1}{\Free}{\Free}{\Free} 		& q :&= \Pattern{0}{0}{\Free}{\Free} 			& r :&= \Pattern{\Free}{\Free}{1}{1} \\
   f_1[p]  &= 1 					& f_1[q]  &= 0 						& f_1[r]  &= v_1+v_2 \\
   f_2[p]  &= 1\cdot v_4 				& f_2[q]  &= 0 						& f_2[r]  &= v_1\cdot1 \\
   f_3[p]  &= 0 					& f_3[q]  &= \overline 0\cdot v_4 			& f_3[r]  &= \overline{v_1}\cdot1 \\
   f_4[p]  &= \overline{v_3} 				& f_4[q]  &= \overline{v_3}				& f_4[r]  &= 0 \\
   p > F[p] &= \Pattern{1}{\Free}{0}{\Free} 		& q = F[q] &= \Pattern{0}{0}{\Free}{\Free} 		& r \not\geq F[r] &= \Pattern{\Free}{\Free}{\Free}{0}
 \end{align*}
 \caption{
 Three examples of how to compute the image of a subspace under $F$ for the running example.
 $p$ and $q$ satisfy the the condition of Thm.~\ref{thm:1} and are therefore trap spaces.
 }\label{fig:03}
\end{figure}

\subsection{Applications}\label{Methods:Applications}
\subsubsection{Application 1: Model Reduction}\label{Methods:Application1}
Let $T\subseteq \StateSpace$ be a trap set.
Denote by $\SmallestSubspace T\in\SubSpaces$ the \emph{smallest subspace} that contains $T\subseteq\StateSpace$, i.e.,
$U_{\SmallestSubspace T}:=\{v\in V \mid \forall x,y\in T: x(v)=y(v)\}$ and $\SmallestSubspace{T}(v):=x(v)$ for $x\in T$ arbitrary.
Note that, in general, the smallest subspace that contains $T$ is a superset of $T$, i.e., $\StateSpace[\SmallestSubspace T]\supseteq T$.
A natural model reduction technique is based on the observation that for any trap set $T\subseteq \StateSpace$, 
the transitions of any path with an initial state $x_1\in T$ only depend on the reduced system $(V_p,F_p)$ with $p:=\SmallestSubspace T$ and
\[V_p := \{v\in V \mid v\not\in U_{p}\}, \;\; F_p := \{f_i[p] \mid f_i\in F,v_i\in V_p\}.\]
Intuitively speaking, the network $(V_p,F_p)$ is obtained by "dividing out" the fixed variables of $p$,
see \cite{Siebert2011SymbolicSteadyStates} for more details.
Note that, by definition, $\SmallestSubspace{\StateSpace[p]}=p$ for any $p\in\SubSpaces$.
In particular, trap spaces $p$ have this property and so they naturally lend themselves for this reduction technique.

\subsubsection{Application 2: Cyclic Attractors}\label{Methods:Application2}
The following theorem is based on the observation that a minimal trap space is either itself a steady state or it contains no steady state at all.

\begin{theorem}
 $|\MinTrapSpaces F\setminus\SteadyStates F|$ is a lower bound on the number of cyclic attractors of $(\StateSpace,\Transition)$.
\end{theorem}
\begin{proof}
 Let $p\in\MinTrapSpaces F\setminus\SteadyStates F$.
 Since $\StateSpace[p]$ is a trap set, it contains an attractor $A\subseteq \StateSpace[p]$.
 If $A=\{x\}$ then $x\in\TrapSpaces F$ such that $x<p$, which contradicts the minimality of $p$.
\end{proof}

Furthermore, since $\SmallestSubspace{\StateSpace[p]}=p$ for $p\in \TrapSpaces F$, we may conclude that
some $v\in V\setminus U_{\SmallestSubspace{\StateSpace[p]}}$ must be involved in the cyclic behavior.
In our running example, the subspace $p=\Pattern00{\Free}{\Free}$ is a minimal trap space and therefore contains only cyclic attractors
in which either $v_3$ or $v_4$ or both oscillate, see Fig.~\ref{fig:02}.

\subsection{The Prime Implicant Graph}\label{Methods:PrimeImplicantGraph}
In this section we propose a method for computing trap spaces.
The idea is to translate the task into a hypergraph problem in which trap spaces are represented by a sets of arcs that satisfy certain constraints.
We consider a directed hypergraph in which each arc corresponds to a minimal size implicant of $f_i$ or $\overline{f_i}$, for some $v_i\in V$.
Minimal size implicants are also called prime implicants, see e.g. \cite{Crama2011BooleanFunctions}.
Here, we define the following slight variation:
For $c\in\mathbb B$, a \emph{$c$-prime implicant} of a non-constant expression $f$ is a subspace $p\in \SubSpaces$ satisfying $f[p]=c$, and $f[q]\neq c$ for all $q>p$.
For a constant $f_i=c$ we define that $p$ with $U_p:=\{v_i\}$ and $p(v_i)=c$ is its single prime implicant.
In the running example, $\Pattern11{\Free}{\Free}$ satisfies $f_1[\Pattern11{\Free}{\Free}]=1$.
But it is not a $1$-prime implicant of $f_1$, because $\Pattern1{\Free}{\Free}{\Free}>\Pattern11{\Free}{\Free}$ and $f_1[\Pattern1{\Free}{\Free}{\Free}]=1$.

The set of all prime implicants of a Boolean network is denoted by $\PrimeImplicants=\PrimeImplicants(V,F)$ and consists of all $(p,c,v_i)\in \SubSpaces\times\mathbb B\times V$
such that $p$ is a $c$-prime implicant of $f_i$.
The \emph{prime implicant graph} is the directed hypergraph $(\mathcal N, \mathcal A)$ where $\mathcal{N}=\mathcal N(V):=\{p\in \SubSpaces \mid |U_p|=1\}$ 
consists of all subspaces with $|U_p|=1$ (corresponding to literals in propositional logic).
To define the arcs $\mathcal A=\mathcal A(V,F)\subset 2^{\mathcal N}\times 2^{\mathcal N}$ we observe that a subspace $p$ can be written (uniquely) as the intersection
between $k:=|U_p|$ subspaces $p_1,\dots, p_k$ such that $|U_{p_i}|=1$ for all $i$.

The arcs $\mathcal A=\mathcal A(V,F)\subset 2^{\mathcal N}\times 2^{\mathcal N}$ are defined by the mapping
\[\alpha:\PrimeImplicants\rightarrow 2^{\mathcal N}\times 2^{\mathcal N}, \quad (p,c,v_i)\mapsto(\{p_1,\dots,p_k\}, \{q\})\]
where 
\begin{itemize}
 \item[(1)] $p_1,\dots,p_k$ is the decomposition of $p$ into $k=|U_p|$ subspaces $p_i$ with $|U_{p_i}|=1$.
 \item[(2)] $q$ is defined by $U_q:=\{v_i\}$ and $q(v_i):=c$.
\end{itemize}
The prime implicant graph has exactly one arc for every prime implicant, i.e., 
$\mathcal A:=\{\alpha(p,c,v_i) \mid (p,c,v_i)\in \PrimeImplicants\}$.
The \emph{head} of an arc $a=(\{p_1,\dots,p_k\}, \{q\})$ is denoted by $H(a):=q$, and its \emph{tail} by $T(a):=p$ where $p$ is the intersection of $p_1,\dots,p_k$.
The prime implicant graph of the running example is given in Fig.~\ref{fig:04}.

\begin{figure}
 \begin{align*}
  (\Pattern{1}{\Free}{\Free}{\Free},1,v_1) 	& \mapsto a_1 &
  (\Pattern{\Free}1{\Free}{\Free},1,v_1)	& \mapsto a_2 &
  (\Pattern00{\Free}{\Free},0,v_1)		& \mapsto a_3 \\
  (\Pattern{1}{\Free}{\Free}1,1,v_2)		& \mapsto a_4 &
  (\Pattern{0}{\Free}{\Free}{\Free},0,v_2) 	& \mapsto a_5 &
  (\Pattern{\Free}{\Free}{\Free}0,0,v_2) 	& \mapsto a_6 \\
  (\Pattern{0}{\Free}{\Free}1,1,v_3)		& \mapsto a_7 &
  (\Pattern{1}{\Free}{\Free}{\Free},0,v_3)	& \mapsto a_8 &
  (\Pattern{\Free}{\Free}{\Free}0,0,v_3)	& \mapsto a_9 \\
  (\Pattern{\Free}{\Free}0{\Free},1,v_4)	& \mapsto a_{10} &
  (\Pattern{\Free}{\Free}1{\Free},0,v_4)	& \mapsto a_{11} &
 \end{align*}
 \begin{center}(a)\end{center}
 \begin{center}\includegraphics[width=.6\linewidth]{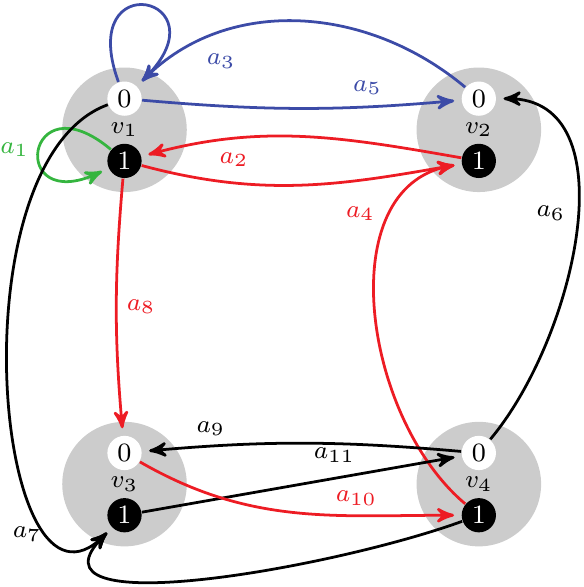}\end{center}
 \begin{center}(b)\end{center}
 \begin{align*}
   &A=\{a_1\} 			& &H(A)=\Pattern1\Free\Free\Free 	& &D_1=\{a_1,a_4,a_8,a_{10}\} \\
   &B=\{a_1,a_8\} 		& &H(B)=\Pattern1\Free0\Free 		& &D_2=\{a_2,a_4,a_8,a_{10}\} \\
   &C=\{a_1,a_8,a_{10}\}	& &H(C)=\Pattern1\Free01 		& &D_3=\{a_1,a_2,a_4,a_8,a_{10}\}\\
   &H(D_1)=1101 		& &H(D_2=1101 				& &H(D_3)=1101
 \end{align*}
 \begin{center}(c)\end{center}
 \caption{
 (a) The complete set $\PrimeImplicants$ of prime implicants of the example network, given in $(p,c,v_i)$ notation.
 (b) The prime implicant graph $(\mathcal N,\mathcal A)$ of the running example.
 The eight nodes are grouped into four pairs that belong to the same variable $v_i$.
 Gray discs represent the groups and are not elements of $\mathcal N$.
 Nodes that correspond to positive literals are drawn in black and negated literals in white.
 Hyperarcs are represented by several arcs having a common arrowhead.
 The colors indicate three of the five stable and consistent arc sets.
 (c) The non-empty stable and consistent arc sets and their induced trap spaces.
 }\label{fig:04}
\end{figure}

\subsection{Prime Implicants and Trap Spaces}
Now we establish a relationship between subsets of $\mathcal A$ and trap spaces.
To do so we need the notions of \emph{consistency} and \emph{stability}.
A subset $A\subseteq\mathcal A$ is \emph{consistent} if for all $a_1,a_2\in A$ and $v\in U_{H(a_1)}\cap U_{H(a_2)}$ it holds that $H(a_1)(v)=H(a_2)(v)$.
If $A=\{a_1,\dots,a_m\}\subseteq\mathcal A$ is consistent then the intersection between $H(a_1),\dots,H(a_m)$ is non-empty,
called the \emph{induced subspace} of $A$ and denoted by $H(A)$.
For the special case $A=\emptyset$ we define $H(A)$ by $U_{H(A)}:=\emptyset$.
A subset $A\subseteq\mathcal A$ is \emph{stable} if for every $a\in A$ there is a consistent subset $B_a\subseteq A$ such that $T(a)\geq H(B_a)$.
Intuitively, in this case the requirement $T(a)$ for each implication $a\in A$ to become effective is met by some assumptions $B_a$.
The stable and consistent subsets of $\mathcal A$ for the running example and their induced subspaces are given in Fig.~\ref{fig:04}.
The central idea for the computation of $\TrapSpaces F$ is given in the next result:
\begin{theorem}
 $p\in\SubSpaces$ is a trap space if and only if there is a stable and consistent $A\subseteq\mathcal A$ such that $H(A)=p$.
\end{theorem}
\begin{proof}
 The statement of the theorem is true if $U_p=\emptyset$ because $A=\emptyset$ is stable and consistent and $H(A)=p$ by definition.
 
 Hence, assume $U_p\neq\emptyset$. Let $v_i\in U_p$.
 Since $p\geq F[p]$ it follows that $f_i[p]=p(v_i)$ and hence that there is a $p(v_i)$-prime implicant $q_i$ of $f_i$ that satisfies $q_i\geq p$.
 The set $A:=\{\alpha(q_i,p(v_i),v_i) \mid v_i\in U_p\}$ is, by definition, consistent and satisfies $H(A)=p$.
 But it is also stable because $T(a)\geq p$ for all $a\in A$.
 Let $\emptyset\neq A\subseteq\mathcal A$ be stable and consistent.
 Then there exists $a\in A$ s.t. $H(a)=H(A)(v_i)$ for all $v_i\in U_{H(A)}$.
 Hence $H(A)\geq F[H(A)]$ and $H(A)\in \TrapSpaces F$.
\end{proof}

The following corollary points out that extremal (i.e., minimal or maximal) stable and consistent arc sets correspond to extremal trap spaces.
As a partial order on the subsets of $\mathcal A$ we use set-inclusion.
\begin{corollary}\label{cor:01}
 (1) If $p\in \MaxTrapSpaces F$ then there is a non-empty minimal stable and consistent $A\subseteq\mathcal A$ s.t. $H(A)=p$.
 (2) If $p\in \MinTrapSpaces F$ then there is a maximal stable and consistent $A\subseteq\mathcal A$ s.t. $H(A)=p$.
\end{corollary}
Note that the inverse relationship (maximal arc sets induce minimal trap spaces and the other way around) stems from the fact that the order on $\SubSpaces$ considers the
free variables while the order on arc sets is in terms of fixed variables.
Note that stable and consistent arc sets are a generalization of the so-called
\emph{self-freezing circuits} which were described and studied in \cite{Fogelman1985BooleanNetworks,Kauffman1993Origins}.
These circuits are based on \emph{canalizing effects} of $F$ which correspond to prime implicants $(p,c,v_i)\in\PrimeImplicants$ that satisfy $|U_p|=1$.
Self-freezing circuits therefore correspond to trap spaces whose stable and consistent $A\subseteq\mathcal A$ satisfy $|U_{T(a)}|=1$ for all $a\in A$.

\subsection{Computation of Trap Spaces}\label{Methods:Computation}
In this section we formulate a 0-1 optimization problem to compute all extremal stable and consistent $A\subseteq\mathcal A$ and therefore all extremal trap spaces.
To solve it in practice, we suggest using solvers for \emph{Integer Linear Programming} (ILP) or \emph{Answer Set Programming} (ASP) and give a reformulation of the constraints
in terms of each language.

As a preliminary step the set $\PrimeImplicants(V,F)$ has to be enumerated, see e.\,g.\ \cite{Jabbour2014EnumeratingPIs}.
Although the number of prime implicants may grow exponentially with the number of variables that the expression depends on, see e.\,g.\ \cite{Crama2011BooleanFunctions},
we found that for typical biological models these dependencies are so small that the enumeration of $\PrimeImplicants$ is negligible compared
with finding consistent and stable arc sets.

We now formulate the 0-1 optimization problem.
For every arc $a=(p,c,v_i)\in\PrimeImplicants$ we introduce a variable $x_a \in \{0,1\}$ that indicates whether or not $a$ is a member of
the set $A\subseteq\mathcal A$ that we want to compute.
We denote these variables by $X:=\{x_a \mid a\in \PrimeImplicants\}$.
In addition, we introduce for every $v_i\in V$ two variables $y^0_i,y^1_i \in \{0,1\}$ that indicate whether $v_i$ is fixed in the trap space and,
if so, what value it takes.
We denote them by $Y:=\{y^c_i \mid c\in\mathbb B,v_i\in V\}$.
For any $c\in\mathbb B$, we require $y_i^c =1$ if and only if $v_i\in U_{H(A)}$ and $H(A)(v_i)=c$.
To encode this requirement, we use the logical constraints
\begin{equation}\tag{C1}
 y_i^c  \enspace \Longleftrightarrow \enspace \bigvee_{a\in B_i^c}x_a, \qquad \text{ for all } c \in \mathbb B, v_i \in V,
\end{equation}
where $B_i^c:=\{a\in \mathcal A \mid \{v_i\}=U_{H(a)},H(a)(v_i)=c \}$ denotes the arcs inducing $v_i$ to take the value $c$,
and $\Rightarrow$ and $\vee$ are the standard logical connectives for implication and disjunction.
To enforce that the set $A:=\{a\in\mathcal A \mid x_a=1\}$ is stable and consistent we add the following constraints (C2) resp.\ (C3):
\begin{align}\tag{C2}
  x_a & \Rightarrow  y_i^{T(a)(v_i)},  & \text{ for all }  & a\in\mathcal A, v_i\in U_{T(a)}, \\
\tag{C3}
  \overline{y_i^0} & \; \vee \; \overline{y_i^1}, & \text{ for all } & v_i \in V.
\end{align}
The maximal resp. minimal stable and consistent $A\subseteq\mathcal A$ correspond to all solutions of
\begin{equation}
 \tag{0-1 max} \text{max} \sum_{x_a\in X}x_a,\quad  \text{ s.t. (C1),(C2),(C3)},
\end{equation}
respectively
\begin{equation}
 \tag{0-1 min} \text{min} \sum_{x_a\in X}x_a,\quad  \text{ s.t. (C1),(C2),(C3) and $1\leq \sum_{x_a\in X}x_a$},
\end{equation}
where the additional constraint for the minimal solutions forbids the empty set.

\subsubsection{ILP Formulation}
We now reformulate the above constraints by linearizing the logical operators.
The resulting problem can be solved with standard ILP solvers.
The first constraints, (ILP1), enforce that $y_i^c$ if there is no arc targeting it:
\begin{equation}\tag{ILP1}
  y_i^c  \leq \displaystyle\sum_{a\in B_i^c}x_a \quad\text{ and }\quad\forall a\in B_i^c: x_a\leq y_i^c,
\end{equation} 
for every $c\in \mathbb B$ and $v_i\in V$.
The next constraints enforce stability and consistency, respectively:
\begin{align}\tag{ILP2}
  x_a & \leq y_i^{T(a)(v_i)}, &\text{ for all } & a\in\mathcal A, v_i\in U_{T(a)}, \\
\tag{ILP3}
  y_i^0 & +y_i^1\leq 1,  & \text{ for all }  & v_i\in V.
\end{align}

A \textsc{python} implementation using \textsc{gurobi} \cite{Gurobi} is available from \cite{BoolNetFixpoints}.
Since ILP solvers usually do not return multiple solutions we suggest to iteratively add \emph{no-good-cuts} that make the current solution infeasible but do not otherwise affect the feasible region,
until there are no more solutions.
Suppose $s:X\cup Y\rightarrow \{0,1\}$ is a solution.
Cuts for the maximization and minimization of (0-1), respectively, are

\begin{align*}
 & 1\leq\sum_{x_a\in G}x_a, 		& \text{where }	& G:=\{x_a\in X \mid s(x_a)=0\}\\
 & \sum_{x_a\in E}x_a\leq|E|-1,		& \text{where }	& E:=\{x_a\in X \mid s(x_a)=1\}.
\end{align*}

\subsubsection{ASP Formulation}
We now reformulate the constraints (C1)-(C3) as an answer set program, see e.g. \cite{Potassco}.
To encode the arcs $\mathcal A$ we introduce two ternary predicates, \texttt{head(v,c,ID)} and the \texttt{tail(v,c,ID)},
where \texttt v refers to some $v\in V$, \texttt c to a value in $c\in\mathbb B$ and \texttt{ID} is an index that determines to which arc a tail or a head belongs.
Each $a\in\mathcal A$ is then translated into a number of so-called \emph{facts} by stating all the tail elements and the unique head element it consists of.
For example, an arc $a_3=(\Pattern00{\Free}{\Free},0,v_1)$ of the running example in Fig.~\ref{fig:04} becomes
\begin{center}
 \texttt{head(v1,0,a3). tail(v1,0,a3). tail(v2,0,a3).}
\end{center}
Note that the index \texttt{a3} links the data together.
In the generate-and-test fashion of defining ASP problems we generate all possible subsets of $\mathcal A$ 
and introduce an unary predicate \texttt{x(ID)} to indicate whether the arc with index \texttt{ID} belongs to the solution $A\subseteq\mathcal A$ or not.
It encodes the variables $x_a\in X$ in the formulation of the (0-1) problem above.
\begin{center}
 \texttt{\{x(ID) : head(v,c,ID)\}.}
\end{center}
The ASP formulation does not require the auxiliary variables $Y$ and hence can do without (C1).
The stability (C2) is translated into the filter (ASP2) which forbids the existence of an arc (with identifier \texttt{ID1}) such that one of its tail nodes
is not also the head of another arc (with identifier \texttt{ID2}).
The consistency (C3) is translated into the filter (ASP3).
It forbids that two arcs (with identifiers \texttt{ID1} and \texttt{ID2}) target the same \texttt v but at different values \texttt 0 and \texttt 1.

\begin{equation}\tag{ASP2}
 \texttt{:- x(ID1), tail(v,c,ID1), not x(ID2): head(v,c,ID2).}\\
\end{equation}

\begin{equation}\tag{ASP3}
 \texttt{:- x(ID1), x(ID2), head(v,1,ID1), head(v,0,ID2).}\\
\end{equation}

To compute multiple solutions is built into ASP solvers and the solving collection \textsc{potassco} \cite{Potassco} also features the option to find set-inclusion minimal or maximal
solutions with respect to the predicates that are true.
For the problem at hand we optimize over the predicate \texttt{x(ID)}.
To forbid the empty solution when minimizing we use the additional filter "\texttt{\;:- \{x(\_)\} 0.\;}".
A \textsc{python} implementation using the solving collection \textsc{potassco} is available at \cite{BoolNetFixpoints}.

\subsubsection{Benchmark}
\begin{figure*}
 \centering
 \makebox[\textwidth][c]{\includegraphics[height=6cm]{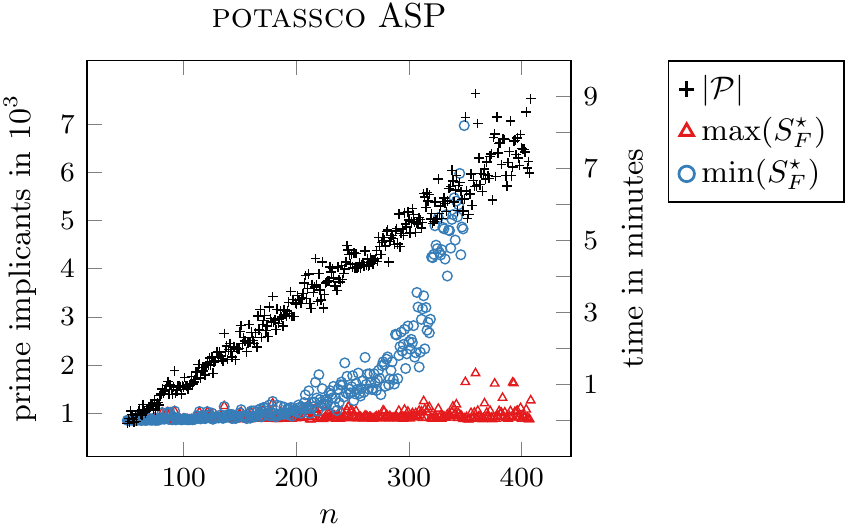}\includegraphics[height=6cm]{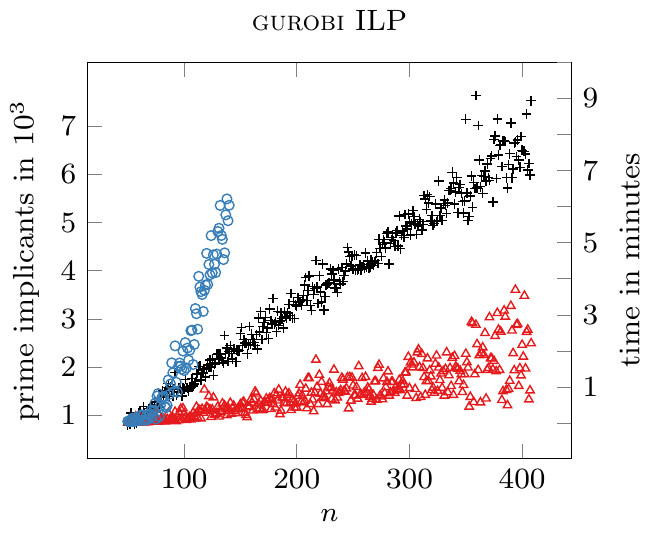}}\\
 (a) Running times for the different solvers\vspace{.5cm}
 
 \makebox[\textwidth][c]{\includegraphics[height=6cm]{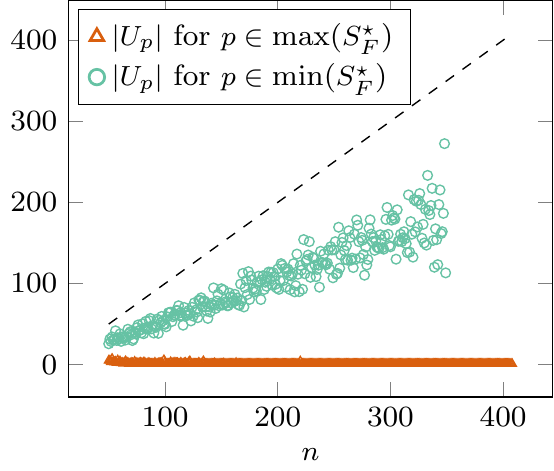}\hspace{.5cm}\includegraphics[height=6.07cm]{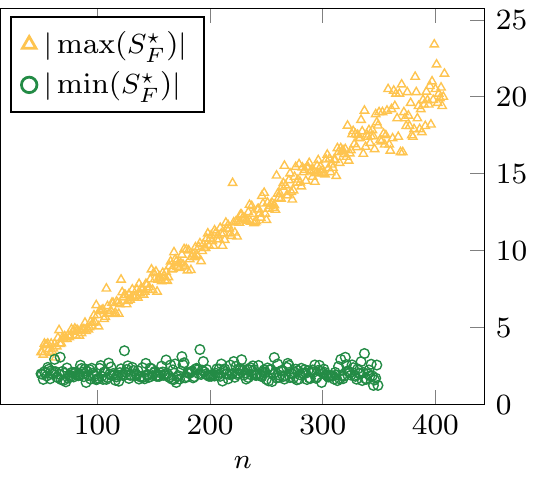}}\\
 
 (b) Number and sizes of trap spaces
 \caption{
 (a) The average running times of \textsc{potassco} and \textsc{gurobi} for computing $\MinTrapSpaces F$ and $\MaxTrapSpaces F$.
 Each plot also contains the total number of prime implicants $|\PrimeImplicants|$.
 Discontinued plots indicate a time-out of the respective solver.
 (b) The plot on the left shows how many variables $|U_p|$ are (on average) fixed in a minimal and maximal trap space $p$.
 The dashed line is a plot of $f(n)=n$ and therefore indicates how close a trap spaces is to being a steady state.
 The plot on the right shows the average number of minimal and maximal trap spaces.
 }\label{fig:05}
\end{figure*}

To test the efficiency and scalability of the ASP and ILP formulations we created random Boolean networks and recorded the time necessary to compute
$\MinTrapSpaces F$ and $\MaxTrapSpaces F$ with each solver.
For the results to be reproducible we decided to use the function \texttt{generateRandomNKNetwork} of the R package \textsc{boolnet} \cite{Mussel2010BoolNet}.
It takes the parameters $n$ and $k$ which specify the number of variables $n=|V|$ and number of variables $k$ that each $f\in F$ depends on.
In addition to $n$ and $k$ there are different configurations for the \emph{topology},
the \emph{linkage} and the \emph{bias} in the generation of the truth tables of the Boolean expressions.
We decided to generate networks whose in-degree follows a Poisson distribution (\texttt{topology="homogeneous"}) with mean and variance equal to $k=3$ and
left the other parameters at their default values (\texttt{linkage="uniform"} and \texttt{functionGeneration="uniform"}).
This setup allows variables with high in-degrees (hubs) as well as low in-degrees (e.g. inputs, cascades, outputs) that frequently appear in models of biological networks.

For each $n$, starting from $n=50$, we generated $50$ networks and called \textsc{gurobi} and \textsc{potassco} to find all minimal and maximal trap spaces of each.
For each network we recorded the
number of prime implicants,
number of minimal and maximal trap spaces,
average number of fixed variables in the trap spaces
and the average time for each solver to compute them.

We treated each solver and whether minimal or maximal trap spaces are computed, as an independent computation and allowed a time limit of $10$ minutes for each.
When a computation failed (time-out or out of memory) $25$ times for the same $n$ we removed it from the benchmark loop.
To solve the ASP problems, we used \textsc{gringo} version 3.0.5 and \textsc{clasp} version 3.1.1 with the configuration 
\texttt{--dom-pref=32}, \texttt{--heu=domain}, \texttt{--dom-mod=6} (subset minimality) or \texttt{--dom-mod=7} (subset maximality).
To solve the ILP problems we used used the \textsc{python} interface to \textsc{gurobi}, version 5.0.
The executions were ran on a Linux desktop PC with 30 GB RAM and eight CPUs with 3.00GHz.
The results are given in Fig.~\ref{fig:05}.

The first observation is that the \textsc{potassco} formulation performs better than \textsc{gurobi} formulation for both maximal and minimal trap spaces.
The time required to compute maximal trap spaces appears to grow linearly while the one for minimal trap spaces grows exponentially.
A reason for why our ILP formulation is less efficient than the ASP formulation might be that the ILP solutions are constructed iteratively,
by forbidding the last found solution, while the ASP solutions are computed in a single execution of \textsc{clasp}.

The two plots at the bottom of Fig.~\ref{fig:05} give some statistical information about the number of trap spaces and their sizes in terms of fixed variables.
The number of fixed variables in a maximal or minimal trap space appears to be a fixed fraction of $n$ and the number of maximal trap spaces
is constant while the number of minimal trap spaces grows linearly with $n$.

\section{Application to a MAPK pathway model}\label{Results}
We computed the extremal trap spaces for a network that models the influence of the MAPK pathway on cancer cell fate decisions, as published in \cite{Grieco2013MAPK}.
It consists of $53$ variables that represent signaling proteins, genes and phenomenological components like \emph{proliferation} or \emph{apoptosis}.
We found that there are $18$ minimal trap spaces, $12$ of which are steady states.
Hence, following Application~2 in Sec.~\ref{Methods:Applications},
there are at least six cyclic attractors whose properties can be comprehensively investigated using the six corresponding reduced models.
The trap spaces $\MinTrapSpaces F\setminus\SteadyStates F$ as well as one reduced model are given in Fig.~\ref{fig:06}.
Next we ask how many attractors $A\subseteq \StateSpace[p]$ there are for each $p\in\MinTrapSpaces F\setminus\SteadyStates F$.
We used \textsc{bns} \cite{Dubrova2011FindingAttractors} to compute all attractors of $(S,\SyncTransition)$ and \textsc{genysis} \cite{Garg2008FindingAttractors} to compute the attractors
of $(S,\AsyncTransition)$.

\begin{figure*}
 \centering
 \includegraphics[width=\linewidth]{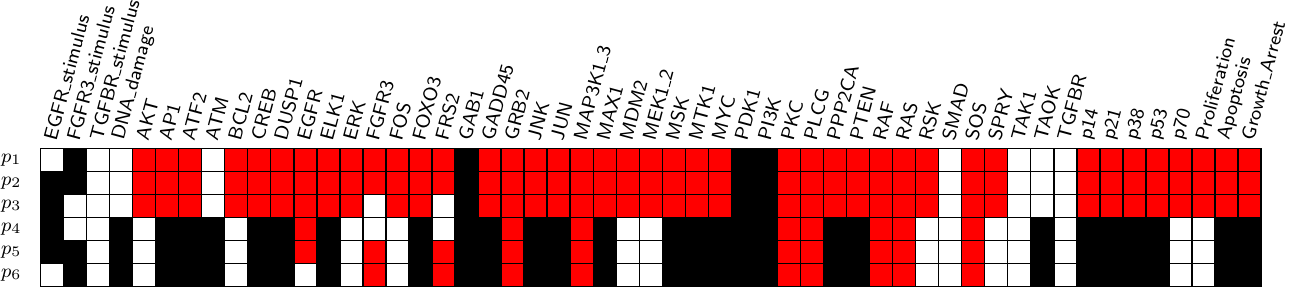}
 \begin{center}(a)\end{center}
 \raisebox{-0.5\height}{\includegraphics[width=.7\linewidth]{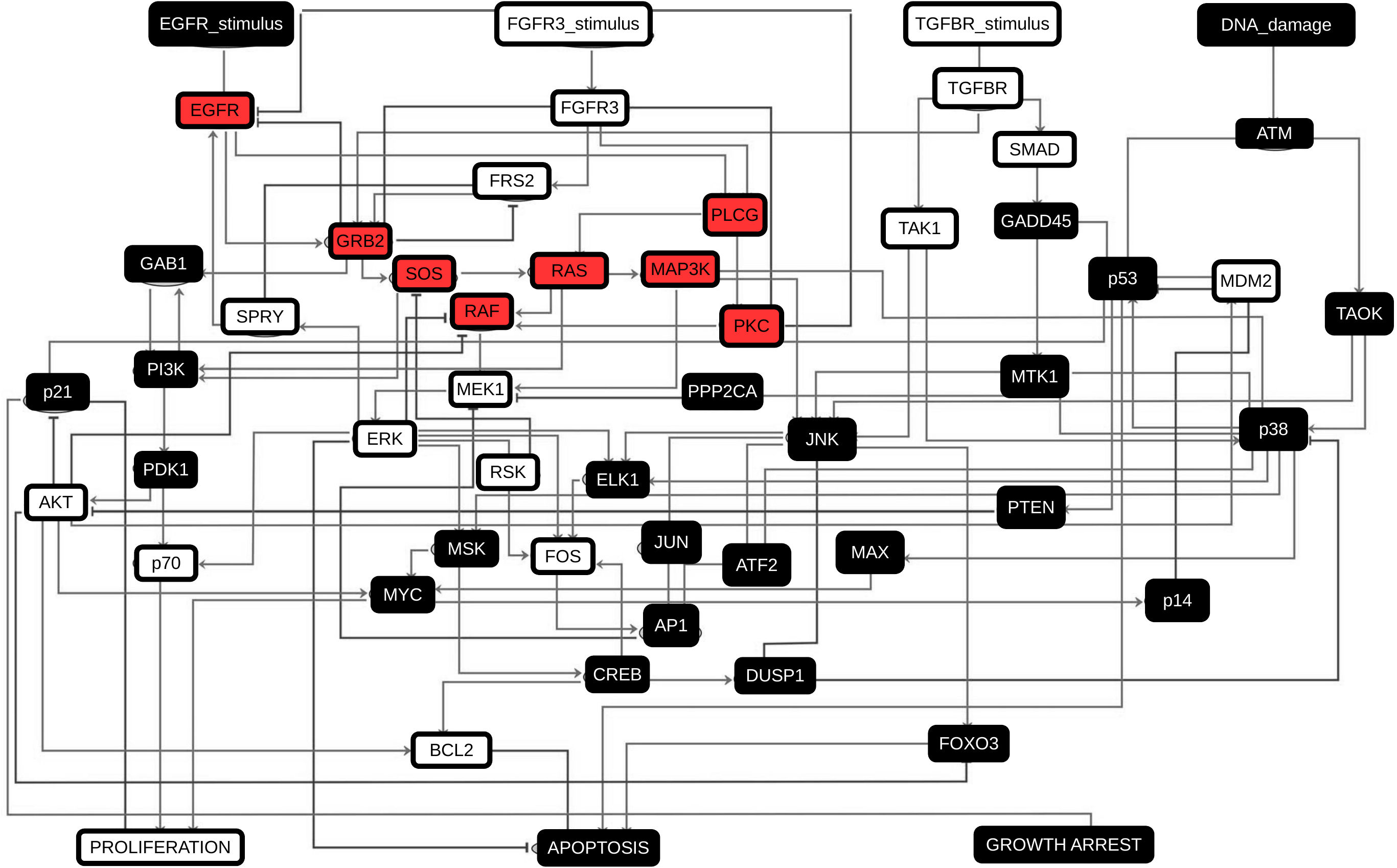}} 
 \hspace{5mm}
 \raisebox{-0.5\height}{\includegraphics[width=.22\linewidth]{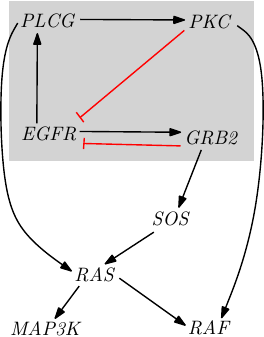}} 
 \begin{center}(b)\end{center}
 \caption{
 (a) The six trap spaces $\MinTrapSpaces F\setminus\SteadyStates F=\{p_1,\dots,p_6\}$ in terms of fixed variables in black and white and free variables in red.
 (b) We mapped the fixed and free variables of $p_4$ onto the interaction graph of the MAPK network.
 The corresponding reduced network consists of $|V|-|U_{p_4}|=53-45=8$ variables and is given on the right.
 }\label{fig:06}
\end{figure*}

It turns out that for the asynchronous update, each trap space $p\in\MinTrapSpaces F$ contains exactly one attractor.
In addition, the attractor $A\subseteq \StateSpace[p]$ satisfies $\SmallestSubspace A=p$ for each $p$, i.e., the fixed and free variables of $p$ correspond exactly to
the variables that are steady and oscillating in its attractor $A$.
We obtained this result by computing the attractors of the reduced networks $(V_p,F_p)$ of each $p\in\MinTrapSpaces{F}$ with \textsc{genysis}.
The question whether there is an attractor outside of $\cup_{p\in\MinTrapSpaces{F}}\StateSpace[p]$ can not be answered with this method because
the unreduced MAPK network is too large for \textsc{genysis}.

All attractors of the synchronous transition graph can be computed with \textsc{bns}, even for the unreduced MAPK network.
We found that there are 28 cyclic attractors.
Contrary to the asynchronous case, some minimal trap spaces do contain more than one attractor and $\SmallestSubspace A<p$ does frequently hold.
For example, $p_1$ contains four cyclic attractors and although $|U_{p_1}|=12$ between $19$ and $24$ variables are steady in the attractors.
In addition, $18$ attractor are not contained in any minimal trap space.

We then computed the maximal trap spaces and found that $|\MaxTrapSpaces{F}|=9$.
The MAPK network has four input variables $v\in V$, namely \emph{EGFR\_stimulus}, \emph{FGFR3\_stimulus}, \emph{TGFBR\_stimulus} and \emph{DNA\_damage},
that satisfy $f_v=v$.
Each input generates two maximal trap spaces $p^0,p^1$ defined by $U_{p^c}=\{v\}$ and $p^c(v)=c$, for $c\in\mathbb B$.
Eight of the maximal trap spaces are therefore explained by the inputs.
The additional $q\in\MaxTrapSpaces{F}$ is defined by $U_q:=\{\mathit{PI3K},\mathit{GAB1}\}$ and $q(\mathit{PI3K})=q(\mathit{GAB1})=1$.
A summary of the commitments of the different maximal trap spaces to steady states and cyclic attractors is given in Tab.~\ref{tab:01}.

\begin{table*}
 \caption{
 For each of the 9 maximal trap spaces we counted how many steady states and cyclic attractors in the synchronous and asynchronous transition graphs they contain.
 The first eight columns correspond to the inputs at either 0 or 1 and the last column refers to the maximal trap space in which \emph{PI3K} and \emph{GAB1} are
 both steady at 1.
 Note that $(\mathit{TGFBR\_stimulus}=1)$ alone ensures that the system will end up in a steady state (zero cyclic attractors reachable) and that most steady states (10 out of 12)
 are inside the \emph{PI3K, GAB1} trap space. 
 }\label{tab:01}
 \centering
 \makebox[\textwidth][c]{\begin{tabular}{lC{7mm}C{7mm}|C{7mm}C{7mm}|C{7mm}C{7mm}|C{7mm}C{7mm}|C{7mm}}
  \multirow{ 2}{*}{maximal trap spaces}
  &\multicolumn{2}{c|}{\scriptsize{\emph{EGF\_stim}}}&
  \multicolumn{2}{c|}{\scriptsize\emph{FGF\_stim}}&
  \multicolumn{2}{c|}{\scriptsize\emph{TGF\_stim}}&
  \multicolumn{2}{c|}{\scriptsize\emph{DNA\_dmg}}&
  \scriptsize\emph{PI3k,GAB1}\\
 &$0$&
  $1$&
  $0$&
  $1$&
  $0$&
  \cellcolor{lightgray}$1$&
  $0$&
  $1$&
  $1$\\\hline\hline
  Steady states&						8&	4&	8&	4&	4&\cellcolor{lightgray}8&	6&	6&	10\\
  Cyclic attractors in $(\StateSpace, \SyncTransition)$&	24&	4&	17&	11&	28&\cellcolor{lightgray}0&	22&	6&	19\\
  Cyclic attractors in $(\StateSpace, \AsyncTransition)$&	2&	4&	2&	4&	6&\cellcolor{lightgray}0&	3&	3&	6\\
 \end{tabular}}
\end{table*}

\section{Discussion}\label{Discussion}
In this paper we propose to use trap spaces for the prediction of a network's attractors and for model reduction.
We propose a novel, optimization-based method for computing trap spaces that uses ILP or ASP solvers.
Its input is the prime implicant graph rather than the full state space (which is necessarily exponential in $V$).
The method can be extended from Boolean to multi-valued networks by generalizing the notion of prime implicants from Boolean to multi-valued expressions.
Thm.~\ref{thm:1}, the characterization of trap spaces, holds not only for synchronous or asynchronous but for \emph{any} update rule (see \cite{Gershenson2004DoTheyMatter}
for other update rules) and also stochastic simulations.
Note also that the prime implicant graph is similar to the process hitting graphs in \cite{Pauleve2014ProcessHitting}, certain prime implicants correspond for example to actions,
but different in that it attempts to capture a single model exactly rather than to approximate a family of models.
We are not aware that process hitting has been used to compute trap spaces.

The benchmark demonstrates that the prime implicant method for computing trap spaces scales well with the number of variables of a network.
All maximal and minimal trap spaces, including steady states, of a network with between 300 and 400 variables and poisson-distributed connectivity with $k=3$,
are computed with an average running time on the order of minutes.
A comparison with a brute force approach for computing trap spaces, by enumerating subspaces,
and the circuit-enumeration approach presented in \cite{Zanudo2013StableMotifs} would be desirable to further assess the efficiency of this approach.

The results for the MAPK network suggest that the usefulness of minimal trap spaces in approximating attractors depends a lot on the update rule.
For the asynchronous update, we observe that
(1) each $p\in\MinTrapSpaces{F}$ contains exactly one attractor $A$ and that
(2) $U_p$ corresponds exactly to the steady variables of $A$.
For the synchronous update, none of these properties are met and, additionally, there are $18$ attractors that lie outside of any minimal trap space.
It appears to us that the trap space method for predicting attractors is, in general, less precise for synchronous transition graphs.

Regarding the use of maximal trap spaces it is worth noting that although they usually contain several attractors,
they are likely to play an important role in decision making and processes like cell differentiation.
An example is the observation that $\mathit{TGFBR}=1$ eliminates the possibility of sustained oscillations and guarantees that the system ends up in a steady state.

Currently, we are working on a method that combines computing trap spaces with reduction techniques and \emph{model checking} to decide, for a given network,
whether properties (1) and (2) hold.

\renewcommand{\abstractname}{Acknowledgements}
\section*{Acknowledgements}
  We thank S. Videla, M. Ostrowski and T. Schaub of University of Potsdam for their help with the ASP formulation.

\bibliographystyle{plain}
\bibliography{klarner}

\end{document}